\documentclass[a4paper]{amsart}
\usepackage{amssymb}
\usepackage[mathcal]{euscript}
\usepackage[cmtip,all]{xy}
\usepackage{color}

\newcommand{\bydef}{:=}
\newcommand{\tr}{\mathrm{tr}} %trace
\newcommand{\norm}{\mathrm{n}}  %norm
\newcommand{\id}{\mathrm{id}}%identity map

\newcommand{\bi}{\mathbf{i}}%imaginary units
\newcommand{\bj}{\mathbf{j}}
\newcommand{\bk}{\mathbf{k}}
\newcommand{\bl}{\mathbf{l}}

\newcommand{\cC}{\mathcal{C}} %algebras, vector spaces
\newcommand{\cU}{\mathcal{U}}
\newcommand{\cV}{\mathcal{V}}
\newcommand{\cK}{\mathcal{K}}
\newcommand{\cQ}{\mathcal{Q}}
\newcommand{\cS}{\mathcal{S}}
\newcommand{\cO}{\mathcal{O}}
\newcommand{\cA}{\mathcal{A}}
\newcommand{\cB}{\mathcal{B}}

\DeclareMathOperator{\CD}{\mathfrak{CD}}%Cayley-Dickson
\newcommand{\Cl}{\mathfrak{Cl}} %Clifford algebra

\newcommand{\frg}{{\mathfrak g}}%Lie algebras
\newcommand{\frsu}{{\mathfrak{su}}}
\newcommand{\frso}{{\mathfrak{so}}}
\newcommand{\tri}{\mathfrak{tri}}

\newcommand{\NN}{\mathbb{N}}

\newcommand{\RR}{\mathbb{R}}
\newcommand{\CC}{\mathbb{C}}
\newcommand{\HH}{\mathbb{H}}
\newcommand{\OO}{\mathbb{O}}
\newcommand{\PP}{\mathbb{P}}
\newcommand{\FF}{\mathbb{F}}
\newcommand{\KK}{\mathbb{K}}
\newcommand{\LL}{\mathbb{L}}
\newcommand{\chr}[1]{\mathrm{char}\,#1}

\DeclareMathOperator{\Mat}{\mathrm{Mat}}
\DeclareMathOperator{\End}{\mathrm{End}}
\DeclareMathOperator{\Aut}{\mathrm{Aut}}
\newcommand{\Ort}{\mathrm{O}}
\newcommand{\SOrt}{\mathrm{SO}}  
\newcommand{\PSOrt}{\mathrm{PSO}}  
\newcommand{\Spin}{\mathrm{Spin}}

\newcommand{\subo}{_{\bar 0}}

\newtheorem{theorem}{Theorem}[section]
\newtheorem{proposition}[theorem]{Proposition}
\newtheorem{lemma}[theorem]{Lemma}
\newtheorem{corollary}[theorem]{Corollary}

\theoremstyle{definition}
\newtheorem{definition}[theorem]{Definition}

\newtheorem{examples}[theorem]{Examples}

\theoremstyle{remark}
\newtheorem{remark}[theorem]{Remark}

\def\bigstrut{\vrule height 12pt width 0ptdepth 2pt}

\def\hregleta{\hrule height .5pt}
\def\hreglon{\hrule height1pt}
\def\vreglon{\vrule height 12pt width1pt depth 4pt}
\def\vregleta{\vrule width .5pt}
\def\hreglonfill{\leaders\hreglon\hfill}

\def\hregletafill{\leaders\hregleta\hfill}

\begin{document}

\title{Composition algebras}

\author{Alberto Elduque}
\address{Departamento de Matem\'{a}ticas
 e Instituto Universitario de Matem\'aticas y Aplicaciones,
 Universidad de Zaragoza, 50009 Zaragoza, Spain}
\email{elduque@unizar.es}
\thanks{Supported by grants MTM2017-83506-C2-1-P (AEI/FEDER, UE) and E22\_17R (Diputaci\'on General de Arag\'on)}



%\date{}

\begin{abstract}
This paper is devoted to survey composition algebras and some of their applications.

After overviewing the classical algebras of quaternions and octonions, both  unital composition algebras (or Hurwitz algebras) and symmetric composition algebras will be dealt with. Their main properties, as well as their classifications, will be reviewed. Algebraic triality, through the use of symmetric composition algebras, will be considered too.
\end{abstract}

\maketitle

Unital composition algebras are the analogues of the classical algebras of the real and complex numbers, quaternions and octonions, but the class of composition algebras have been recently enriched with new algebras, mainly the so called symmetric composition algebras, which play a nice role in understanding the triality phenomenon in dimension $8$. The goal of this paper is to survey the main definitions and results on these algebras.

Section \ref{se:intro} will review the discovery of the real algebra of quaternions by Hamilton, will survey some of the applications of quaternions to deal with rotations in three and four dimensional euclidean spaces, and will move to octonions, discovered shortly after Hamilton's breakthrough.

Section \ref{se:Hurwitz} will be devoted to the classification of the unital composition algebras, also termed Hurwitz algebras. This is achieved by means of the Cayley-Dickson doubling process, which mimics the way in which Graves and Cayley constructed the octonions by doubling the quaternions. Hurwitz itself, in 1898, proved that a positive definite quadratic form over the real numbers allows composition if and only if the dimension is restricted to $1,2,4$ or $8$. These are the possible dimensions of Hurwitz algebras over arbitrary fields, and of finite-dimensional non-unital composition algebras. The case of Hurwitz algebras with isotropic norm will be given some attention, as this is instrumental in defining Okubo algebras later on, although the model of Zorn's vector matrices will not be touched upon.

In Section \ref{se:symmetric}, symmetric composition algebras will be defined. In dimension $>1$ these are non-unital composition algebras, satisfying the extra condition of `associativity of the norm': $n(x*y,z)=n(x,y*z)$. The interest lies in dimension $8$, where these algebras split in two disjoint families: para-Hurwitz algebras and Okubo algebras. The existence of the latter Okubo algebras justifies the introduction of the symmetric composition algebras. 

Formulas for triality are simpler if one uses symmetric composition algebras, instead of the classical Hurwitz algebras. This will be the subject of Section \ref{se:triality}. Triality is a broad subject. In Projective Geometry, there is duality relating points and hyperplanes. Given a vector space of dimension $8$, endowed with a quadratic form $q$ with maximal Witt index, the quadric of isotropic vectors $Q=\{\FF v: v\in V,\ q(v)=0\}$ in projective space $\PP(V)$ contains points, lines, planes and `solids', and there are two kinds of `solids'. Geometric triality relates points and the two kinds of solids in a cyclic way. This goes back to Study \cite{Study} and Cartan \cite{Cartan}. Tits \cite{Tits} showed that there are two different types of
geometric trialities, one of them is  related to
octonions (or para-octonions) and the exceptional group $G_2$, while the other is
related to the classical groups of type $A_2$, unless the
characteristic is $3$. The algebras hidden behind this second type are the Okubo algebras.
From the algebraic point of view, triality relates the natural and spin representations of the spin group on an eight-dimensional quadratic space, that is the three irreducible representations corresponding to the outer vertices of the Dynkin diagram $D_4$. Here we will consider lightly the local version of triality, which gives a very symmetric construction of Freudenthal's Magic Square.

\section{Quaternions and Octonions}\label{se:intro}

It is safe to say that the history of composition algebras starts with the discovery of the real quaternions by Hamilton.

%\bigskip

\subsection{Quaternions} \null\quad

Real and complex numbers are in the toolkit of any scientist. Complex numbers corresponds to vectors in a Euclidean plane, so that addition and multiplication by real scalars are the natural ones for vectors. The length (or modulus, or norm) of the product of two complex numbers is the product of the lengths of the factors. In this way, multiplication by a norm 1 complex numbers is an isometry, actually a rotation, of the plane, and this allows us to identify the group of rotations of the Euclidean plane, that is, the special orthogonal group $\SOrt_2(\RR)$, with the set of norm one complex numbers (the unit circle):
\[
\SOrt_2(\RR)\simeq\{z\in\CC : \lvert z\rvert=1\}\simeq S^1.
\]

In 1835, William Hamilton posed himself the problem of extending the domain of complex numbers to a system of numbers `of dimension $3$'. He tried to find a multiplication, analogous to the multiplication of complex numbers, but in dimension $3$, that should respect the `law of moduli':
\begin{equation}\label{eq:moduli}
\lvert z_1z_2\rvert =\lvert z_1\rvert\lvert z_2\rvert .
\end{equation}
That is, he tried to get a formula for
\[
(a+b\bi + c\bj)(a'+b'\bi+c'\bj)=??
\]
and assumed $\bi^2=\bj^2=-1$. The problem is how to define $\bi\bj$ and $\bj\bi$.

With some `modern' insight, it is easy to see that this impossible. A \emph{nonassociative} (i.e., not necessarily associative) \emph{algebra} over a field $\FF$ is just a vector space $\cA$ over $\FF$, endowed with a bilinear map (the multiplication) $m:\cA\times\cA\rightarrow \cA$, $(x,y)\mapsto xy$. We will refer to the algebra $(\cA,m)$ or simply $\cA$ if no confusion arises. The algebra $\cA$ is said to be a \emph{division algebra} if the left and right multiplications $L_x:y\mapsto xy$, $R_x:y\mapsto yx$, are bijective for any nonzero $x\in\cA$. The `law of moduli' forces the algebra sought for by Hamilton to be a division algebra. But this is impossible.

\begin{proposition}\label{pr:impossible}
There are no real division algebras of odd dimension $\geq 3$.
\end{proposition}
\begin{proof} (See \cite{PeterssonEMS})
Given a real algebra $\cA$ of odd dimension $\geq 3$, and linearly independent elements $x,y\in\cA$ with $L_x$ and $L_y$ bijective, $\det(L_x+tL_y)$ is a polynomial in $t$ of odd degree, and hence it has a real root $\lambda$. Thus the left multiplication $L_{x+\lambda y}$ is not bijective.
\end{proof}

Therefore, Hamilton could not succeed. But, after years of struggle, he found how to overcome the difficulties in dimension $3$ by making a leap to dimension $4$.

In a letter to his son Archibald, dated August 5, 1865, Hamilton explained how vividly he remembered the date: \textbf{October 16, 1843}, when he got the key idea. He wrote:
\begin{quotation}
Every morning in the early part of the above-cited month, on my coming down to breakfast, your (then) little brother William Edwin, and yourself, used to ask me, {``Well, Papa, can you multiply triplets''?} Whereto I was always obliged to reply, with a sad shake of the head: {``No, I can only add and subtract them.''}

But on the 16th day of the same month - which happened to be a Monday, and a Council day of the Royal Irish Academy - I was walking in to attend and preside, and your mother was walking with me, along the Royal Canal, to which she had perhaps driven; and although she talked with me now and then, yet an under-current of thought was going on in my mind, which gave at last a result, whereof it is not too much to say that I felt at once the importance. {An electric circuit seemed to close; and a spark flashed forth,} the herald (as I foresaw, immediately) of many long years to come of definitely directed thought and work, by myself if spared, and at all events on the part of others, if I should even be allowed to live long enough distinctly to communicate the discovery.

Nor could I resist the impulse -unphilosophical as it may have been- to cut with a knife on a stone of Brougham Bridge, as we passed it, the fundamental formula with the symbols, $\bi$, $\bj$, $\bk$; namely,
\[
\bi^2 = \bj^2 = \bk^2 = \bi\bj\bk = -1
\]
which contains the Solution of the Problem, but of course, as an inscription, has long since mouldered away. A more durable notice remains, however, on the Council Books of the Academy for that day (October 16th, 1843), which records the fact, that I then asked for and obtained leave to read a Paper on Quaternions, at the First General Meeting of the session: which reading took place accordingly, on Monday the 13th of the November following.
\end{quotation}

Hamilton realized that the product $\bi\bj$ had to be linearly independent of $1,\bi,\bj$, and hence he defined the real algebra of \emph{quaternions} as:
\[
\HH=\RR 1\oplus\RR \bi\oplus\RR\bj\oplus\RR\bk
\]
with multiplication determined by 
\begin{gather*}
 \bi^2=\bj^2=\bk^2=-1,\\
 \bi\bj=-\bj\bi=\bk,\quad \bj\bk=-\bk\bj=\bi,\quad \bk\bi=-\bi\bk=\bj.
\end{gather*}

Some properties of this algebra are summarized here:

\begin{itemize}
\item For any $q_1,q_2\in \HH$, $\lvert q_1q_2\rvert=\lvert q_1\rvert\lvert q_2\rvert$ $\forall q_1,q_2\in\HH$, where for $q=a+b\bi+c\bj+c\bk$, $\lvert a+b\bi+c\bj+d\bk\rvert^2=a^2+b^2+c^2+d^2$ is the standard Euclidean norm on $\HH\simeq \RR^4$.

\item $\HH$ is an associative division algebra, but we loose the commutativity valid for the real and complex numbers.

Therefore $S^3\simeq\{q\in\HH: \lvert q\rvert=1\}$ is a (Lie) group.
This implies the parallelizability of the three-dimensional sphere $S^3$.

\item $\HH$ splits as $\HH=\RR 1\oplus \HH_0$, where $\HH_0=\RR \bi\oplus\RR \bj\oplus\RR \bk$ is the subspace of \emph{purely imaginary quaternions}.

Then, for any $u,v\in\HH_0$:
   
\[
uv=-u\bullet v+u\times v
\]
where $u\bullet v$ and $u\times v$ denote the usual scalar and cross products in $\RR^3\simeq\HH_0$.

\item For any $q=a1+u\in \HH$ ($u\in\HH_0$), $q^2=(a^2-u\bullet u)+2au$, so
\begin{equation}\label{eq:H_quadratic}
q^2-\tr(q)q+\norm(q)1=0
\end{equation}
with $\tr(q)=2a$ and $\norm(q)=\lvert q\rvert^2=a^2+u\bullet u$. That is, $\HH$ is a \emph{quadratic algebra}.

\item The map $q=a+u\mapsto \overline{q}=a-u$ is an involution, with $q+\overline{q}=\tr(q)=2a\in\RR$ and
 $q\overline{q}=\overline{q} q=\norm(q)=\lvert q\rvert^2\in\RR$.

\item $\HH=\CC\oplus \CC \bj\,\simeq \CC^2$ is a two-dimensional vector space over $\CC$. The 
multiplication is then given by:
\begin{equation}\label{eq:HasCDC}
(p_1+p_2\bj)(q_1+q_2\bj)=(p_1q_1-\overline{q_2}p_2)+(q_2p_1+p_2\overline{q_1})\bj
\end{equation}
for any $p_1,p_2,q_1,q_2\in\CC$.

\end{itemize}

%\bigskip

\subsection{Rotations in three (and four) dimensional space} \null\quad

Given a norm $1$ quaternion $q$, there is an angle $\alpha\in [0,\pi]$ and a norm $1$ imaginary quaternion such that $q=(\cos\alpha)1+(\sin\alpha)u$.

Consider the linear map:
\[
\begin{split}
\varphi_q:\HH_0&\longrightarrow \HH_0,\\
x&\mapsto qxq^{-1}=qx\overline{q}.
\end{split}
\]
Complete $u$ to an orthonormal basis $\{u,v,u\times v\}$. A simple computation gives:
\[
\begin{aligned}
\varphi_q(u)&=quq^{-1}=u\qquad\text{(as $uq=qu$),}\\[2pt]
\varphi_q(v)&=\bigl((\cos\alpha)1+(\sin\alpha)u\bigr)v((\cos\alpha)1-(\sin\alpha)u)\\
 &=\bigl((\cos\alpha)v+(\sin\alpha)u\times v\bigr)((\cos\alpha)1-(\sin\alpha)u)\\
 &=(\cos^2\alpha)v+2(\cos\alpha\sin\alpha)u\times v-(\sin^2\alpha)(u\times v)\times u\\
 &=(\cos 2\alpha)v+(\sin 2\alpha)u\times v,\\[2pt]
\varphi_q(u\times v)&=. . .=-(\sin 2\alpha)v+(\cos 2\alpha)u\times v.
\end{aligned}
\]
Thus the coordinate matrix of  $\varphi_q$ relative to the basis $\{u,v,u\times v\}$ is
\[
\begin{pmatrix} 1&0&0\\ 0&\cos 2\alpha&-\sin 2\alpha\\ 0&\sin 2\alpha&\cos 2\alpha \end{pmatrix}
\]

In other words, $\varphi_q$ is a rotation around the semi axis $\RR^+u$ of angle $2\alpha$, and hence the map
\[
\begin{split}
\varphi: S^3\simeq \{q\in\HH: \lvert q\rvert =1\}&\longrightarrow \SOrt_3(\RR),\\
  q\quad&\mapsto\quad \varphi_q
\end{split}
\]
is a surjective (Lie) group homomorphism with $\ker\varphi=\{\pm 1\}$. We thus obtain the isomorphism 
\[
S^3/_{\{\pm 1\}}\simeq \SOrt_3(\RR)
\]
Actually, the group $S^3$ is the universal cover of $\SOrt_3(\RR)$.

Therefore, we get that rotations can be identified with conjugation by norm $1$ quaternions modulo $\pm 1$. The outcome is that it is quite easy now to compose rotations in three-dimensional Euclidean space, as it is enough to multiply norm $1$ quaternions: $\varphi_p\varphi_q=\varphi_{pq}$. From here one can deduce very easily the 1840 formulas by Olinde Rodrigues \cite{Rodrigues} for the composition of rotations.

\smallskip

But there is more about rotations and quaternions. 

For any $p\in \HH$ with $\norm(p)=1$, 
the left (resp. right) multiplication $L_p$ (resp. $R_p$) by $p$ is an isometry, due to the multiplicativity of the norm. Using \eqref{eq:H_quadratic} it follows that the characteristic polynomial of $L_p$ and $R_p$ is $\left(x^2-\tr(p)x+1\right)^2$ and, in particular, the determinant of the multiplication by $p$ is $1$, so both $L_p$ and $R_p$ are rotations.

Now, if $\psi$ is a rotation in  $\RR^4\simeq\HH$, $a=\psi(1)$ is a norm $1$ quaternion, and
\[
L_{\overline{a}}\psi(1)=\overline{a} a=\norm(a)=1,
\]
so the composition $L_{\overline{a}}\psi$ is actually a rotation in $\RR^3\simeq \HH_0$.
Hence, there is a norm $1$ quaternion $q\in\HH$ such that
\[
\overline{a}\psi(x)=qxq^{-1}
\]
for any $x\in\HH$. That is, for any $x\in \HH$,
\[
\psi(x)=(aq)xq^{-1}
\]
It follows that the map
\[
\begin{split}
 \Psi: S^3\times S^3&\longrightarrow \SOrt_4(\RR),\\
   (p,q)&\mapsto \psi_{p,q} : x\mapsto pxq^{-1}
\end{split}
\]
is a surjective (Lie) group homomorphism with $\ker\Psi=\{\pm (1,1)\}$. We thus obtain the isomorphism
\[
S^3\times S^3/_{\{\pm (1,1)\}}\simeq \SOrt_4(\RR)
\]
and from here we get the isomorphism $\SOrt_3(\RR)\times \SOrt_3(\RR)\simeq \PSOrt_4(\RR)$.

Again, this means that it is 
 quite easy to compose rotations in four-dimensional space, as 
it reduces to multiplying pairs of norm $1$ quaternions: $\psi_{p_1,q_1}\psi_{p_2,q_2}=\psi_{p_1p_2,q_1q_2}$.

%\bigskip

\subsection{Octonions} \null\quad

In a letter from Graves to Hamilton, dated October 26, 1843, only a few days after the `discovery' of quaternions, Graves writes:

\begin{quotation}
There is still something in the system which gravels me. I have not yet any clear views as to the extent to which we are at liberty arbitrarily to create imaginaries, and to endow them with supernatural properties.

{If with your alchemy you can make three pounds of gold, why should you stop there?}
\end{quotation}

Actually, as we have seen in \eqref{eq:HasCDC}, the algebra of quaternions is obtained 
 by doubling suitably the field of complex numbers:
$
\HH=\CC\oplus\CC \bj.
$

Doubling again we get the \emph{octonions} (Graves--Cayley):
\[
\OO=\HH\oplus\HH \bl=\textrm{span}_{\RR}\left\{1,\bi,\bj,\bk,\bl,\bi\bl,\bj\bl,\bk\bl\right\}
\]
with multiplication mimicked from \eqref{eq:H_quadratic}:
\[
(p_1+p_2\bl)(q_1+q_2\bl)=(p_1q_1-\overline{q_2}p_2)+(q_2p_1+p_2\overline{q_1})\bl
\]
 and usual norm:
$
\norm(p_1+p_2\bl)=\norm(p_1)+\norm(p_2)
$,
for $p_1,p_2,q_1,q_2\in\HH$.

This was already known to Graves, who wrote a letter to Hamilton on December 26, 1843 with his discovery of what he called \emph{octaves}. Hamilton promised to announce Graves' discovery to the Irish Royal Academy, but did not do it in time. In 1845, independently, Cayley discovered the octonions and got the credit. Octonions are also called \emph{Cayley numbers}.

\smallskip

Some properties of this new algebra of octonions are summarized here:

\begin{itemize}

\item The norm is multiplicative: $\norm(xy)=\norm(x)\norm(y)$, for any $x,y\in\OO$.

\item $\OO$ is a division algebra, and it is neither commutative nor associative!

But it is \emph{alternative}, that is, any two elements generate an associative subalgebra.

A Theorem by Zorn \cite{Zorn} asserts that
the only finite-dimensional real alternative division algebras are  $\RR$, $\CC$, $\HH$ and $\OO$.
And hence, as proved by Frobenius \cite{Frobenius}, the only such associative algebras are $\RR$, $\CC$ and $\HH$.

\item The seven-dimensional Euclidean spehre $S^7\simeq \{ x\in \OO: \norm(x) =1\}$ is not a group (associativity fails), but it constitutes the most important example of a  \emph{Moufang loop}.

\item As for $\HH$, for any two \emph{imaginary octonions} $u,v\in\OO_0=
\textrm{span}_\RR\left\{\bi,\bj,\bk,\bl,\bi\bl,\bj\bl,\bk\bl\right\}$ we have:
\[
uv=-u\bullet v+u\times v.
\]
for the usual scalar product $u\bullet v$ on $\RR^7\simeq\OO_0$, and where $u\times v$ defines the usual cross product in $\RR^7$. This satisfies the identity $(u\times v)\times v=(u\bullet v)v-(v\bullet v)u$, for any $u,v\in\RR^7$.

\item $\OO$ is again a \emph{quadratic algebra}:  $x^2-\tr(x)x+\norm(x)1=0$ for any $x\in \OO$, where $\tr(x)=x+\overline{x}$ and $\norm(x)=x\overline{x}=\overline{x} x$, where for $x=a1+u$, $a\in\RR$, $u\in\OO_0$, $\overline{x}=a1-u$.
\end{itemize}

\medskip

And, as it happens for quaternions, octonions are also present in many interesting geometrical situations, here we mention just a few:

\begin{itemize}
\item The groups $\Spin_7$ and $\Spin_8$ (universal covers of $\SOrt_7(\RR)$ and $\SOrt_8(\RR)$) can be described easily in terms of octonions.

\item The fact that $\OO$ is a division algebra implies the parallelizability of the seven-dimensional sphere $S^7$. Actually,
$S^1$, $S^3$ and $S^7$ are the only parallelizable spheres \cite{Adams,BottMilnor,Kervaire}.

\item
The six-dimensional sphere can be identified with the set of norm $1$ imaginary units: $S^6\simeq\{x\in\OO_0: \norm(x) =1\}$, and it is endowed with an \emph{almost complex structure}, inherited from the multiplication of octonions.

$S^2$ and $S^6$ are the only spheres with such structures \cite{BorelSerre}.

\item Contrary to what happens in higher dimensions, projective planes need not be desarguesian. The simplest example of a non-desarguesian projective plane is the octonionic projective plane $\OO P^2$.
\end{itemize}

\bigskip

A lot of material on the work of Hamilton can be found at the URL: 
\begin{center}
\texttt{https://www.maths.tcd.ie/pub/HistMath/People/Hamilton/}
\end{center} 
by David R.~Wilkins. The interested reader may consult \cite{Numbers} or \cite{ConwaySmith} for complete expositions on quaternions and octonions.

\bigskip

\section{Unital composition algebras}\label{se:Hurwitz}

Composition algebras constitute a generalization of the classical algebras of the real, complex, quaternion and octonion numbers.

A quadratic from $\norm:V\rightarrow \FF$ on a vector space $V$ over a field $\FF$ is said to be \emph{nondegenerate} if so is its polar form:
\[
\norm(x,y)\bydef \norm(x+y)-\norm(x)-\norm(y),
\]
that is, if its radical $V^\perp\bydef\{v\in V:\norm(v,V)=0\}$ is trivial.
Moreover, $\norm$ is said to be \emph{nonsingular} if either it is nondegenerate or it satisfies that the dimension of $V^\perp$ is $1$ and $\norm(V^\perp)\neq 0$. The last possibility only occurs over fields of characteristic $2$.

\begin{definition}\label{df:composition}
A \emph{composition algebra} over a field $\FF$ is a triple $(\cC,\cdot,\norm)$ where
\begin{itemize}
\item $(\cC,\cdot)$ is a nonassociative algebra, and
\item $\norm:\cC\rightarrow\FF$ is a nonsingular quadratic form which is \emph{multiplicative}, that is,
\begin{equation}\label{eq:nxy}
\norm(x\cdot y)=\norm(x)\norm(y)
\end{equation}
for any $x,y\in\cC$.
\end{itemize}
\end{definition}

The unital composition algebras are called \emph{Hurwitz algebras}.

For simplicity, we will refer usually to the composition algebra $\cC$.

Our goal in this section is to prove that Hurwitz algebras are quite close to $\RR$, $\CC$, $\HH$ and $\OO$.

By linearization of \eqref{eq:nxy} we obtain:
\begin{equation}\label{eq:nxy_linear}
\begin{split}
&\norm(x\cdot y,x\cdot z)=\norm(x)\norm(y,z),\\
&\norm(x\cdot y,t\cdot z)+\norm(t\cdot y,x\cdot  z)=\norm(x,t)\norm(y,z),
\end{split}
\end{equation}
for any $x,y,z,t\in\cC$.

\begin{proposition}\label{pr:Hurwitz_properties}
Let $(\cC,\cdot,\norm)$ be a Hurwitz algebra.
\begin{itemize}
\item Either $\norm$ is nondegenerate or $\chr\FF=2$ and $\cC$ is isomorphic to the ground field $\FF$ (with norm $\alpha\mapsto \alpha^2$).

\item The map $x\mapsto \overline{x}\bydef \norm(1,x)1-x$ is an involution. That is, $\overline{\overline{x}}=x$ and $\overline{x\cdot y}=\overline{y}\cdot \overline{x}$ for any $x,y\in\cC$. This involution is referred to as the \emph{standard conjugation}.

\item If $*$ denotes the conjugation of a linear endomorphism relative to $\norm$ (i.e., $\norm\bigl(f(x),y\bigr)=\norm\bigl(x,f^*(y)\bigr)$ for any $x,y$), then for the left and right multiplications by elements $x\in\cC$ we have $L_x^*=L_{\overline{x}}$ and $R_x^*=R_{\overline{x}}$.

\item Any $x\in\cC$ satisfies the \emph{Cayley-Hamilton equation}:
\[
x^{\cdot 2}-\norm(x,1)x+\norm(x)1=0.
\]

\item $(\cC,\cdot)$ is an alternative algebra: $x\cdot(x\cdot y)=x^{\cdot 2}\cdot y$ and $(y\cdot x)\cdot x=y\cdot x^{\cdot 2}$ for any $x,y\in\cC$.
\end{itemize}
\end{proposition}
\begin{proof}
Plug $t=1$ in \eqref{eq:nxy_linear} to get
\[
\norm(x\cdot y,z)=\norm\bigl(y,(\norm(x,1)1-x)z\bigr)=\norm(y,\overline{x}\cdot z)
\]
and symmetrically we get $\norm(y\cdot x,z)=\norm(y,z\cdot\overline{x})$.

Now, if $\chr\FF=2$ and $\cC^\perp=\FF a$, with $\norm(a)\neq 0$, then for any $x,y\in\cC$, $\norm(a\cdot x,y)=n(a,y\cdot \overline{x})=0$, so $a\cdot x\in\cC^\perp$ and $a\cdot x=f(x)a$ for a linear map $f:\cC\rightarrow \FF$. But $\norm(a)\norm(x)=\norm(a\cdot x)=f(x)^2\norm(a)$. Hence $\norm(x)=f(x)^2$ for any $x$ and, by linearization, $\norm(x,y)=2f(x)f(y)=0$ for any $x,y$. We conclude that $\cC=\cC^\perp=\FF 1$. In this case, all the assertions are trivial.

Assuming hence that $\cC^\perp=0$ ($\norm$ is nondegenerate), since $x\mapsto\overline{x}$ is an isometry of order $2$ (reflection relative to $\FF 1$) we get
\[
\norm(\overline{x\cdot y},z)=\norm(x\cdot y,\overline{z})=\norm(x,\overline{z}\cdot\overline{y})=
\norm(z\cdot x,\overline{y})=\norm(z,\overline{y}\cdot \overline{x})
\]
for any $x,y,z$, whence $\overline{x\cdot y}=\overline{y}\cdot\overline{x}$.

Finally, for any $x,y,z$, using again \eqref{eq:nxy_linear},
\[
\norm(x)\norm(y,z)=\norm(x\cdot y,x\cdot z)=\norm(\overline{x}\cdot(x\cdot y),z)
\]
so that, using left-right symmetry:
\begin{equation}\label{eq:barxxy}
\overline{x}\cdot(x\cdot y)=\norm(x)y =(y\cdot x)\cdot\overline{x}.
\end{equation} 
With $y=1$ this gives $\overline{x}\cdot x=\norm(x)1$ and hence the Cayley-Hamilton equation. On the other hand, $\overline{x}\cdot(x\cdot y)=\norm(x)y=(\overline{x}\cdot x)\cdot y$ and this shows $x\cdot (x\cdot y)=x^{\cdot 2}\cdot y$. Symmetrically we get $(y\cdot x)\cdot x=y\cdot x^{\cdot 2}$.
\end{proof}

\subsection{The Cayley-Dickson doubling process and the Generalized Hurwitz Theorem} \null\quad

Let $(\cC,\cdot,\norm)$ be a Hurwitz algebra, and assume that $\cQ$ is a proper unital subalgebra of $\cC$ such that the restriction of $\norm$ to $\cQ$ is nondegenerate. Our goal is to show that in this case $\cC$ contains also a subalgebra obtained `doubling' $\cQ$, in a way similar to the construction of $\HH$ from two copies of $\CC$, or the construction of $\OO$ from two copies of $\HH$.

By nondegeneracy of $\norm$, $\cC=\cQ\oplus \cQ^\perp$. Pick $u\in \cQ^\perp$ with $\norm(u)\neq 0$, and let $\alpha=-\norm(u)$. As $1\in\cQ$, $\norm(u,1)=0$ and hence $\overline{u}=-u$ and $u^{\cdot 2}=\alpha 1$ by the Cayley-Hamilton equation (Proposition \ref{pr:Hurwitz_properties}). This also implies that $R_u^2=\alpha\id$, so the right multiplication $R_u$ is bijective.

\begin{lemma}\label{le:Hurwitz}
Under the conditions above, the subspaces $\cQ$ and $\cQ\cdot u$ are orthogonal (i.e., $\cQ\cdot u\subseteq \cQ^\perp$), and the following properties hold for any $x,y\in \cQ$.
\begin{itemize}
\item $x\cdot u=u\cdot\overline{x}$,

\item $x\cdot (y\cdot u)=(y\cdot x)\cdot u$,

\item $(y\cdot u)\cdot x=(y\cdot\overline{x})\cdot u$,

\item $(x\cdot u)\cdot (y\cdot u)=\alpha\overline{y}\cdot x$.
\end{itemize}
\end{lemma}
\begin{proof}
For any $x,y\in\cQ$, $\norm(x,y\cdot u)=\norm(\overline{y}\cdot x,u)\in\norm(\cQ,u)=0$, so $\cQ\cdot u$ is a subspace orthogonal to $\cQ$.

From $\norm(x\cdot u,1)=\norm(u,\overline{x})\in\norm(u,\cQ)=0$, it follows that $\overline{x\cdot u}=-x\cdot u$. But $\overline{x\cdot u}=\overline{u}\cdot\overline{x}=-u\cdot\overline{x}$, whence $x\cdot u=u\cdot\overline{x}$.

Now $x\cdot (y\cdot u)=-x\cdot (\overline{y\cdot u})=-x\cdot (\overline{u}\cdot \overline{y})$ and this is equal, because of \eqref{eq:barxxy}, to $u\cdot(\overline{x}\cdot\overline{y})=u\cdot(\overline{y\cdot x})=(y\cdot x)\cdot u$.

In a similar vein $(y\cdot u)\cdot x=-(y\cdot\overline{u})\cdot x=(y\cdot\overline{x})\cdot u$, and $(x\cdot u)\cdot (y\cdot u)=-(\overline{x\cdot u})\cdot (y\cdot u)=\overline{y}\bigl((x\cdot u)\cdot u\bigr)=\overline{y}\cdot (x\cdot u^{\cdot 2})=\alpha\overline{y}\cdot x$.
\end{proof}

Therefore, the subspace $\cQ\oplus\cQ\cdot u$ is also a subalgebra, and the restriction of $\norm$ to it is nondegenerate. The multiplication and norm are given by (compare to \eqref{eq:HasCDC}):
\begin{equation}\label{eq:CD}
\begin{split}
&(a+b\cdot u)\cdot(c+d\cdot u)=
 (a\cdot c+\alpha\overline{d}\cdot b)+(d\cdot a+b\cdot\overline{c})\cdot u,\\
&\norm(a+b\cdot u)=\norm(a)-\alpha\norm(b),
\end{split}
\end{equation}
for any $a,b,c,d\in\cQ$.

Moreover, 
\[
\norm\bigl((a+b\cdot u)\cdot(c+d\cdot u)\bigr)=\norm(a\cdot c+\alpha\overline{d}\cdot b)-\alpha\norm(d\cdot a+b\cdot\overline{c}),
\] 
while on the other hand
\[
\begin{split}
\norm(a+b\cdot u)\norm(c+d\cdot u)
&=\bigl(\norm(a)-\alpha\norm(b)\bigr)\bigl(\norm(c)-\alpha\norm(d)\bigr)\\
&=\norm(a)\norm(c)+\alpha^2\norm(b)\norm(d)-\alpha\bigl(\norm(d)\norm(a)+\norm(b)\norm(c)\bigr)\\
&=\norm(a\cdot c)+\alpha^2\norm(\overline{d}\cdot b)-\alpha\bigl(\norm(d\cdot a)+\norm(b\cdot\overline{c})\bigr).
\end{split}
\]
We conclude that $\norm(a\cdot b,\alpha\overline{d}\cdot b)-\alpha\norm(d\cdot a,b\cdot\overline{c})=0$, or $\norm\bigl(d\cdot (a\cdot c),b\bigr)=\norm\bigl((d\cdot a)\cdot c,b\bigr)$.

The nondegeneracy of the restriction of $\norm$ to $\cQ$ implies then that $\cQ$ is associative! In particular, any proper subalgebra of $\cC$ with nondegenerate restricted norm is associative.

\smallskip

Conversely, given an \emph{associative} Hurwitz algebra $\cQ$ with nondegenerate $\norm$, and a nonzero scalar $\alpha\in\FF$, consider the direct sum of two copies of $\cQ$: $\cC=\cQ\oplus\cQ\cdot u$, with multiplication and norm given by \eqref{eq:CD}, extending those on $\cQ$. The arguments above show that $(\cC,\cdot,\norm)$ is again a Hurwitz algebra, which is said to be obtained by the \emph{Cayley-Dickson doubling process} from $(\cQ,\cdot,\norm)$ and $\alpha$. This algebra is denoted by $\CD(\cQ,\alpha)$.

\begin{remark}\label{re:associative_commutative}
$\CD(\cQ,\alpha)$ is associative if and only if $\cQ$ is commutative. This follows from $x\cdot (y\cdot u)=(y\cdot x)\cdot u$. If the algebra is associative this equals $(x\cdot y)\cdot u$, and it forces $x\cdot y=y\cdot x$ for any $x,y\in \cQ$. The converse is an easy exercise.
\end{remark}

We arrive at the main result of this section.

\begin{theorem}[Generalized Hurwitz Theorem]\label{th:Hurwitz}
Every Hurwitz algebra over a field $\FF$ is isomorphic to one of the following:
\begin{enumerate}
\item The ground field $\FF$.

\item A two-dimensional separable commutative and associative algebra: $\cK=\FF 1\oplus\FF v$, with $v^{\cdot 2}=v+\mu 1$, $\mu\in\FF$ with $4\mu+1\neq 0$, and $\norm(\epsilon+\delta v)=\epsilon^2-\mu\delta^2+2\epsilon\delta$, for $\epsilon,\delta\in\FF$.

\item A \emph{quaternion} algebra $\cQ=\CD(\cK,\beta)$ for $\cK$ as in \textup{(2)} and $0\neq \beta\in\FF$.

\item A \emph{Cayley} (or \emph{octonion}) algebra $\cC=\CD(\cQ,\gamma)$, for $\cQ$ as in \textup{(3)} and $0\neq \gamma\in\FF$.
\end{enumerate}
In particular, the dimension of a Hurwitz algebra is restricted to $1,2,4$ or $8$.
\end{theorem}
\begin{proof}
The only Hurwitz algebra of dimension $1$ is, up to isomorphism, the ground field. If $(\cC,\cdot,\norm)$ is a Hurwitz algebra and $\dim_\FF\cC>1$, there is an element $v\in\cC\setminus \FF 1$ such that $\norm(v,1)=1$ and $\norm\vert_{\FF 1+\FF v}$ is nondegenerate. The Cayley-Hamilton equation shows that $v^{\cdot 2}-v+\norm(v)1=0$, so $v^{\cdot 2}=v+\mu 1$, with $\mu=-\norm(v)$. The nondegeneracy condition is equivalent to the condition $4\mu+1\neq 0$. Then $\cK=\FF 1+\FF v$ is a Hurwitz subalgebra of $\cC$ and, if $\dim_\FF\cC=2$, we are done. 

If $\dim_\FF\cC>2$ we may take an element $u\in\cK^\perp$ with $\norm(u)=-\beta\neq 0$, and hence the subspace $\cQ=\cK\oplus\cK\cdot u$ is a subalgebra of $\cC$ isomorphic to $\CD(\cK,\beta)$. By the previous remark, $\cQ$ is associative (as $\cK$ is commutative), but it fails to be commutative, as $v\cdot u=u\cdot\overline{v}\neq u\cdot v$. If $\dim_\FF\cC=4$, we are done.

Finally, if $\dim_\FF\cC>4$, we may take an element $u'\in\cQ^\perp$ with $\norm(u')=-\gamma\neq 0$, and hence the subspace $\cQ\oplus\cQ\cdot u'$ is a subalgebra of $\cC$ isomorphic to $\CD(\cQ,\gamma)$, which is not associative by Remark \ref{re:associative_commutative}, so it is necessarily the whole $\cC$.
\end{proof}

Note that if $\chr\FF\neq 2$, the restriction of $\norm$ to $\FF 1$ is nondegenerate, so we could have used the same argument for dimension $>1$ in the proof above than the one used for $\dim_\FF\cC>2$. Hence we get:

\begin{corollary}\label{co:Hurwitz_chR_not2}
Every Hurwitz algebra over a field $\FF$ of characteristic not $2$ is isomorphic to one of the following:
\begin{enumerate}
\item The ground field $\FF$.

\item A two-dimensional algebra $\cK=\CD(\FF,\alpha)$ for a nonzero scalar $\alpha$.

\item A \emph{quaternion} algebra $\cQ=\CD(\cK,\beta)$ for $\cK$ as in \textup{(2)} and $0\neq \beta\in\FF$.
\item A \emph{Cayley} (or \emph{octonion}) algebra $\cC=\CD(\cQ,\gamma)$, for $\cQ$ as in \textup{(3)} and $0\neq \gamma\in\FF$.
\end{enumerate}
\end{corollary}

\begin{remark}\label{re:Hurwitz_R}
Over the real field $\RR$, the scalars $\alpha$, $\beta$ and $\gamma$ in Corollary \ref{co:Hurwitz_chR_not2} can be taken to be $\pm 1$. Note that \eqref{eq:HasCDC} and the analogous equation for $\CC$ and $\OO$ give isomorphisms $\CC\cong\CD(\RR,-1)$, $\HH\cong\CD(\CC,-1)$ and $\OO\cong\CD(\HH,-1)$.
\end{remark}

\begin{remark}
Hurwitz \cite{Hurwitz} only considered the real case with a positive definite norm. Over the years this was extended in several ways. The actual version of the Generalized Hurwitz Theorem seems to appear for the first time in \cite{Jac58} (if $\chr\FF\neq 2$) and \cite{vdBS59}.
\end{remark}

The problem of isomorphism between Hurwitz algebras of the same dimension relies on the norms:

\begin{proposition}\label{pr:Hurwitz_isomorphism}
Two Hurwitz algebras over a field are isomorphic if and only if their norms are isometric.
\end{proposition}
\begin{proof}
Any isomorphism of Hurwitz algebras is, in particular, an isometry of the corresponding norms, due to the Cayley-Hamilton equation. The converse follows from Witt's Cancellation Theorem (see \cite[Theorem 8.4]{Elman_et_al}).
\end{proof}

A natural question is whether the restriction of the dimension of a Hurwitz algebra to be $1,2,4$ or $8$ is still valid for arbitrary composition algebras. The answer is that this is the case for finite-dimensional composition algebras.

\begin{corollary}\label{co:1248}
Let $(\cC,\cdot,\norm)$ be a finite-dimensional composition algebras. Then its dimension is either $1$, $2$, $4$ or $8$.
\end{corollary}
\begin{proof}
Let $a\in\cC$ be an element of nonzero norm. Then $u=\frac{1}{\norm(a)}a^{\cdot 2}$ satisfies $\norm(u)=1$. Using the so called \emph{Kaplansky's trick} \cite{Kap53}, consider the new multiplication
\[
x\diamond y=R_u^{-1}(x)\cdot L_u^{-1}(y).
\]
Note that since the left and right multiplications by a norm $1$ element are isometries, we still have $\norm(x\diamond y)=\norm(x)\norm(y)$, so $(\cC,\diamond,\norm)$ is a composition algebra too. But $u^{\cdot 2}\diamond x=u\cdot L_u^{-1}(x)=x=x\diamond u^{\cdot 2}$ for any $x$, so the element $u^{\cdot 2}$ is the unity of $(\cC,\diamond)$ and $(\cC,\diamond,\norm)$ is a Hurwitz algebra, and hence $\dim_\FF\cC$ is restricted to $1,2,4$ or $8$.
\end{proof}

However, contrary to the thoughts expressed in \cite{Kap53}, there are examples of infinite-dimensional composition algebras. For example (see \cite{UrbanikWright60}), let $\varphi:\NN\times\NN\rightarrow\NN$ be a bijection (for instance, $\varphi(n,m)=2^{n-1}(2m-1)$), and let $\cA$ be a vector space over a field $\FF$ of characteristic not $2$ with a countable basis $\{u_n:n\in\NN\}$. Define a multiplication and a norm on $\cA$ by 
\[
u_n\cdot u_m=u_{\varphi(n,m)},\qquad \norm(u_n,u_m)=\frac{1}{2}\delta_{n,m}.
\]
Then $(\cA,\cdot,\norm)$ is a composition algebra. 

In \cite{EP_infinite} one may find examples of infinite-dimensional composition algebras of arbitrary infinite dimension, which are even left unital.

\subsection{Isotropic Hurwitz algebras} \null\quad

Assume now that the norm of a Hurwitz algebra $(\cC,\cdot,\norm)$ represents $0$. That is, there is a nonzero element $a\in\cC$ such that $\norm(a)=0$. This is always the case if $\dim_\FF\cC\geq 2$ and $\FF$ is algebraically closed.

With $a$ as above, take $b\in\cC$ such that $\norm(a,\overline{b})=1$, so that $\norm(a\cdot b,1)=1$. Also $\norm(a\cdot b)=\norm(a)\norm(b)=0$. By the Cayley-Hamilton equation, the nonzero element $e_1\bydef a\cdot b$ satisfies $e_1^{\cdot 2}=e_1$, that is, $e_1$ is an idempotent. Consider too the idempotent  $e_2\bydef 1-e_1=\overline{e_1}$, and the subalgebra $\cK=\FF e_1\oplus \FF e_2\,(\cong\FF\times\FF)$ generated by $e_1$. ($1=e_1+e_2$). 

For any $x\in\cK^\perp$, $x\cdot e_1+\overline{x\cdot e_1}=\norm(x\cdot e_1,1)=\norm(x,\overline{e_1})=\norm(x,e_2)=0$ and, as $\overline{x}=-x$ and $\overline{e_1}=e_2$, we conclude that $x\cdot e_1=e_2\cdot x$, and in the same way, $x\cdot e_2=e_1\cdot x$, for any $x\in\cK^\perp$.

But $x=1\cdot x=e_1\cdot x+e_2\cdot x$, and $e_2\cdot (e_1\cdot x)=(1-e_1)\cdot (e_1\cdot x)=0=e_1\cdot (e_2\cdot x)$. It follows that $\cK^\perp$ splits as $\cK^\perp=\cU\oplus \cV$ with
\[
\begin{split}
\cU&=\{x\in\cC: e_1\cdot x=x=x\cdot e_2,\ e_2\cdot x=0=x\cdot e_1\},\\
\cV&=\{x\in\cC: e_2\cdot x=x=x\cdot e_1,\ e_1\cdot x=0=x\cdot e_2\}.
\end{split}
\]

For any $u\in\cU$, $\norm(u)=\norm(e_1\cdot u)=\norm(e_1)\norm(u)=0$, so $\cU$, and $\cV$ too, are totally isotropic subspaces of $\cK^\perp$ paired by the norm. In particular, $\dim_\FF\cU=\dim_\FF\cV$, and this common value is either $0$, $1$ or $3$, depending on $\dim_\FF\cC$ being $2$, $4$ or $8$. The case of $\dim_\FF\cU=0$ is trivial, and the case of $\dim_\FF\cU=1$ is quite easy (and subsumed in the arguments below). Hence, let us assume that $\cC$ is a Cayley algebra (dimension $8$), so $\dim_\FF\cU=\dim_\FF\cV=3$.

For any $u_1,u_2\in \cU$ and $v\in\cV$, using \eqref{eq:nxy_linear} we get
\[
\begin{split}
&\norm(u_1\cdot u_2,\cK)\subseteq \norm(u_1,\cK\cdot u_2)\subseteq \norm(\cU,\cU)=0,\\
&\norm(u_1\cdot u_2,v)=\norm(u_1\cdot u_2,e_2\cdot v)=-\norm(e_2\cdot u_2,u_1\cdot v)+
 \norm(u_1,e_2)\norm(u_2,v)=0.
 \end{split}
\]
Hence $\cU^{\cdot 2}$ is orthogonal to both $\cK$ and $\cV$, so it must be contained in $\cV$. Also $\cV^{\cdot 2}\subseteq \cU$.

Besides,
\[
\norm(\cU,\cU\cdot\cV)\subseteq \norm(\cU^{\cdot 2},\cV)\subseteq \norm(\cV,\cV)=0=\norm(\cV,\cU\cdot\cV),
\]
so that $\cU\cdot\cV\subseteq (\cU+\cV)^\perp=\cK$, and also $\cV\cdot\cU\subseteq \cK$. But for any $u\in\cU$ and $v\in \cV$, we have
\[
\norm(u\cdot v,e_2)=-\norm(u,e_2\cdot v)=-\norm(u,v),\quad
\norm(u\cdot v,e_1)=-\norm(u,e_1\cdot v)=0,
\]
and the analogues for $v\cdot u$. We conclude that
\[
u\cdot v=-\norm(u,v)e_1,\quad v\cdot u=-\norm(v,u)e_2.
\]

Now, for linearly independent elements $u_1,u_2\in\cU$, let $v\in\cV$ with $\norm(u_1,v)\neq 0=\norm(u_2,v)$. Then the alternative law gives $(u_1\cdot u_2)\cdot v=-(u_1\cdot v)\cdot u_2+u_1\cdot (u_2\cdot v+v\cdot u_2)=-\norm(u_1,v)u_2\neq 0$, so that $u_1\cdot u_2\neq 0$. In particular $\cU^{\cdot 2}\neq 0$, and the same happens with $\cV$.

Consider the trilinear map:
\[
\begin{split}
\cU\times\cU\times\cU&\longrightarrow \FF\\
(x,y,z)\ &\mapsto \norm(x\cdot y,z).
\end{split}
\]
This is  alternating because $x^{\cdot 2}=0$ for any $x\in \cU$ by the Cayley-Hamilton equation, and $\norm(x\cdot y,y)=-\norm(x,y\cdot y)=0$. It is also nonzero, because $\cU^{\cdot 2}\subseteq \cV$, so that $\norm(\cU^{\cdot 2},\cU)=\norm(\cU^{\cdot 2},\cC)\neq 0$.

Fix a basis $\{u_1,u_2,u_3\}$ of $\cU$ with $\norm(u_1\cdot u_2,u_3)=1$ and take $v_1\bydef u_2\cdot u_3$, $v_2\bydef u_3\cdot u_1$, $v_3\bydef u_1\cdot u_2$. Then $\{v_1,v_2,v_3\}$ is the dual basis in $\cV$ relative to the norm, and the multiplication of the basis $\{e_1,e_2,u_1,u_2,u_3,v_1,v_2,v_3\}$ is completely determined. For instance, $v_1\cdot v_2=v_1\cdot (u_3\cdot u_1)=-u_3\cdot (v_1\cdot u_1)=-u_3\cdot\bigl(-\norm(v_1,u_1)e_2)=u_3$, ...

The multiplication table is given in Figure \ref{fig:splitCayley}.

\begin{figure}[h!]
\[ \vbox{\offinterlineskip
\halign{\hfil$#$\enspace\hfil&#\vreglon
 &\hfil\enspace$#$\enspace\hfil
 &\hfil\enspace$#$\enspace\hfil&#\vregleta
 &\hfil\enspace$#$\enspace\hfil
 &\hfil\enspace$#$\enspace\hfil
 &\hfil\enspace$#$\enspace\hfil&#\vregleta
 &\hfil\enspace$#$\enspace\hfil
 &\hfil\enspace$#$\enspace\hfil
 &\hfil\enspace$#$\enspace\hfil&#\vreglon\cr
 &\omit\hfil\vrule width 1pt depth 4pt height 10pt
   &e_1&e_2&\omit&u_1&u_2&u_3&\omit&v_1&v_2&v_3&\omit\cr
 \noalign{\hreglon}
 e_1&&e_1&0&&u_1&u_2&u_3&&0&0&0&\cr
 e_2&&0&e_2&&0&0&0&&v_1&v_2&v_3&\cr
 &\multispan{11}{\hregletafill}\cr
 u_1&&0&u_1&&0&v_3&-v_2&&-e_1&0&0&\cr
 u_2&&0&u_2&&-v_3&0&v_1&&0&-e_1&0&\cr
 u_3&&0&u_3&&v_2&-v_1&0&&0&0&-e_1&\cr
 &\multispan{11}{\hregletafill}\cr
 v_1&&v_1&0&&-e_2&0&0&&0&u_3&-u_2&\cr
 v_2&&v_2&0&&0&-e_2&0&&-u_3&0&u_1&\cr
 v_3&&v_3&0&&0&0&-e_2&&u_2&-u_1&0&\cr
 &\multispan{12}{\hreglonfill}\cr}}
\]
\caption{Multiplication table of the split Cayley algebra}\label{fig:splitCayley}
\end{figure}

The Cayley algebra with this multiplication table is called the \emph{split Cayley algebra} and denoted by $\cC_s(\FF)$.
The subalgebra spanned by $e_1,e_2,u_1,v_1$ is isomorphic to the algebra $\Mat_2(\FF)$ of $2\times 2$ matrices.

We summarize the above arguments in the next result.

\begin{theorem}\label{th:isotropic}
There are, up to isomorphism, only three Hurwitz algebras with isotropic norm: $\FF\times\FF$, $\Mat_2(\FF)$, and $\cC_s(\FF)$.
\end{theorem}

\begin{corollary}\label{th:reaL_field}
The real Hurwitz algebras are, up to isomorphism, the following algebras:
\begin{itemize}
\item the classical division algebras $\RR$, $\CC$, $\HH$, and $\OO$, and

\item the algebras $\RR\times\RR$, $\Mat_2(\RR)$, and $\cC_s(\RR)$.
\end{itemize}
\end{corollary}
\begin{proof}
It is enough to take into account that a nondegenerate quadratic form over $\RR$ is either isotropic or definite. Hence the norm of a real Hurwitz algebra is either isotropic or positive definite (as $\norm(1)=1$).
\end{proof}

\bigskip

\section{Symmetric composition algebras}\label{se:symmetric}

In this section, a new important family of composition algebras will be described.

\begin{definition}
A composition algebra $(\cS,*,\norm)$ is said to be a \emph{symmetric
composition algebra} if $L_x^*=R_x$ for any $x\in \cS$ (that is,
$\norm(x*y,z)=\norm(x,y*z)$ for any $x,y,z\in \cS$).
\end{definition}

\begin{theorem}\label{th:symmetric_comp}
Let $(\cS,*,\norm)$ be a composition algebra. The following conditions are equivalent:
\begin{itemize}
\item[(a)] $(\cS,*,\norm)$ is symmetric.

\item[(b)] For any $x,y\in \cS$, $(x*y)*x=x*(y*x)=\norm(x)y$.
\end{itemize}
The dimension of any symmetric composition algebra is finite, and hence restricted to $1,2,4$, or $8$.
\end{theorem}
\begin{proof}
If $(\cS,*,\norm)$ is symmetric, then for any $x,y,z\in \cS$,
\[
\norm\bigl((x*y)*x,z\bigr)=\norm(x*y,x*z)=\norm(x)\norm(y,z)=\norm\bigl(\norm(x)y,z\bigr)
\]
so that $(x*y)*x-\norm(x)y\in\cS^\perp$. Also 
\[
\begin{split}
\norm\bigl((x*y)*x-\norm(x)y\bigr)
 &=\norm\bigl((x*y)*x\bigr)+\norm(x)^2\norm(y)-\norm(x)\norm\bigl((x*y)*x,y\bigr)\\
 &=2\norm(x)^2\norm(y)-\norm(x)\norm(x*y,x*y)=0,
\end{split}
\]
whence (b), since $\norm$ is nonsingular.

Conversely, take $x,y,z\in \cS$ with $\norm(y)\ne 0$, so that $L_y$ and $R_y$ are bijective, and hence there is an element $z'\in \cS$ with $z=z'*y$. Then:
\[
\norm(x*y,z)=\norm(x*y,z'*y)=\norm(x,z')\norm(y)
  =\norm\bigl(x,y*(z'*y)\bigr)=\norm(x,y*z).
\]
This proves (a) assuming $\norm(y)\ne 0$, but any isotropic element is the sum of two non isotropic elements, so (a) follows.

Finally, we can use a modified version of Kaplansky's trick (see Corollary \ref{co:1248}) as follows. Let $a\in\cS$ be a norm $1$ element and define a new product on $\cS$ by:
\[
x\diamond y\bydef (a*x)*(y*a),
\]
for any $x,y\in\cS$. Then $(\cS,\diamond,\norm)$ is also a composition algebra. Let $e=a^{* 2}$. Then, using (b) we have $e\diamond x=\bigl(a*(a*a)\bigr)*(x*a)=a*(x*a)=x$, and similarly $x\diamond e=x$ for any $x$. Thus $(\cS,\diamond,\norm)$ is a Hurwitz algebra with unity $e$, and hence it is finite-dimensional.
\end{proof}

\begin{remark} 
Condition (b) above implies that $((x*y)*x)*(x*y)=\norm(x*y)x$, but also $((x*y)*x)*(x*y)=\norm(x)y*(x*y)=\norm(x)\norm(y)x$, so that condition (b) already forces the quadratic form $\norm$ to be multiplicative.
\end{remark}

\begin{examples}\cite{Okubo78}]\label{ex:Okubo} \null\quad
\begin{itemize}
\begin{samepage}
\item \textbf{Para-Hurwitz algebras:}\quad Let $(\cC,\cdot,\norm)$ be a Hurwitz algebra and consider the composition algebra $(\cC,\bullet,\norm)$ with the new product given by
\[
    x\bullet y=\overline{x}\cdot\overline{y}.
\]
Then $\norm(x\bullet y,z)=\norm(\overline{x}\cdot\overline{y},z)=\norm(\overline{x},z\cdot y)=\norm(x,\overline{z\cdot y})=\norm(x,y\bullet z)$,
    for any $x,y,z$, so that $(\cC,\bullet,\norm)$ is a symmetric composition algebra. (Note that $1\bullet x=x\bullet 1=\overline{x}=\norm(x,1)1-x$  for any $x$: $1$ is a \emph{para-unit} of $(\cC,\bullet,\norm)$.)
\end{samepage}

\item \textbf{Okubo algebras:}\quad Assume $\chr\FF\ne 3$ (the case of $\chr\FF =3$ requires a different definition), and let $\omega\in\FF$ be a primitive cubic root of $1$. Let $\cA$ be a central simple associative algebra of degree $3$ with trace $\tr$, and let $\cS=\cA_0=\{x\in \cA:\tr(x)=0\}$. For any $x\in\cS$ the quadratic form $\frac{1}{2}\tr(x^2)$ make sense even if $\chr\FF=2$ (check this!). Define now a multiplication and norm on $\cS$ by:
    \[
    \begin{split}
    &x*y=\omega xy-\omega^2 yx-\frac{\omega-\omega^2}{3}\tr(xy)1,\\[6pt]
    &\norm(x)=-\frac{1}{2}\tr(x^2),
    \end{split}
    \]
Then, for any $x,y\in \cS$:
    \[
    \begin{split}
    (x*y)*x&=\omega(x*y)x-\omega^2x(x*y)-
        \frac{\omega-\omega^2}{3}\tr\bigl((x*y)x\bigr)1\\
        &=\omega^2xyx-yx^2-\frac{\omega^2-1}{3}\tr(xy)x -x^2y+\omega xyx+\frac{1-\omega}{3}\tr(xy)x\\
        &\qquad\quad -\frac{\omega-\omega^2}{3}\tr\Bigl((\omega-\omega^2)x^2y\Bigr)1\quad\text{($\tr(x)=0$)}\\
        &=-\bigl(x^2y+yx^2+xyx\bigr)+\tr(xy)x+\tr(x^2y)1\quad\text{($(\omega-\omega^2)^2=-3$).}
    \end{split}
    \]
But if $\tr(x)=0$, then $x^3-\frac{1}{2}\tr(x^2)x-\det(x)1=0$, so
    \[
    x^2y+yx^2+xyx-\bigl(\tr(xy)x+\frac{1}{2}\tr(x^2)y\bigr)\in\FF 1.
    \]
Since $(x*y)*x\in \cA_0$, we have $(x*y)*x=-\frac{1}{2}\tr(x^2)y=x*(y*x)$.

Therefore $(\cS,*,\norm)$ is a symmetric composition algebra.

\smallskip

    In case $\omega\not\in\FF$, take $\KK=\FF[\omega]$ and a central simple associative algebra $\cA$ of degree $3$ over $\KK$ endowed with a $\KK/\FF$-involution of second kind $J$. Then take $\cS=K(\cA,J)_0=\{x\in \cA_0: J(x)=-x\}$ (this is a $\FF$-subspace) and use the same formulas above to define the multiplication and the norm.
\end{itemize}
\end{examples}

\begin{remark}
For $\FF=\RR$, take $\cA=\Mat_3(\CC)$, and then there appears the Okubo algebra $(\cS,*,\norm)$ with $\cS=\frsu_3=\{x\in \Mat_3(\CC): \tr(x)=0,\ x^*=-x\}$ ($x^*$ denotes the conjugate transpose of $x$). This algebra was termed the algebra of \emph{pseudo-octonions} by Okubo \cite{Okubo78}, who studied these algebras and classified them, under some restrictions, in joint work with Osborn \cite{OO81a,OO81b}.

The name \emph{Okubo algebras} was given in \cite{EM_Okubo}. Faulkner \cite{Faulkner88} discovered independently Okubo's construction in a more general setting, related to separable alternative algebras of degree $3$, and gave the key idea for the classification of the symmetric composition algebras in \cite{EM93} ($\chr\FF\neq 2,3$). A different, less elegant, classification was given in \cite{EM91} using that Okubo algebras are \emph{Lie-admissible}.

The name \emph{symmetric composition algebra} was given in \cite[Chapter VIII]{KMRT}.
\end{remark}

\begin{remark}
Given an Okubo algebra, note that for any  $x,y\in \cS$,
\[
\begin{split}
x*y&=\omega xy-\omega^2 yx-\frac{\omega-\omega^2}{3}\tr(xy)1,\\
y*x&=\omega yx-\omega^2 xy-\frac{\omega-\omega^2}{3}\tr(xy)1,
\end{split}
\]
so that
\[
\omega x*y+\omega^2 y*x=(\omega^2-\omega)xy-(\omega+\omega^2)\frac{\omega-\omega^2}{3}\tr(xy)1,
\]
and
\begin{equation}\label{eq:sym_xy}
xy=\frac{\omega}{\omega^2-\omega}x*y+\frac{\omega^2}{\omega^2-\omega}y*x+\frac{1}{3}\norm(x,y)1,
\end{equation}
so the product in $\cA$ is determined by the product in the Okubo algebra.

Also, as noted by Faulkner, the construction above is valid for separable alternative algebras of degree $3$.
\end{remark}

\begin{theorem}[\cite{EM91,EM93}]
Let $\FF$ be a field of characteristic not $3$.
\begin{itemize}
\item If $\FF$ contains a primitive cubic root $\omega$ of $1$, then the symmetric composition algebras of dimension $\geq 2$ are, up to isomorphism, the algebras $(\cA_0,*,\norm)$  for $\cA$ a separable alternative algebra of degree $3$.

    Two such symmetric composition algebras are isomorphic if and only if so are the corresponding alternative algebras.

\item If $\FF$ does not contain primitive cubic roots of $1$, then the symmetric composition algebras of dimension $\geq 2$ are, up to isomorphism, the algebras $\bigl(K(\cA,J)_0,*,\norm\bigr)$ for $\cA$ a separable alternative algebra of degree $3$ over $\KK=\FF[\omega]$, and $J$ a $\KK/\FF$-involution of the second kind.

    Two such symmetric composition algebras are isomorphic if and only if so are the corresponding alternative algebras, as algebras with involution.
\end{itemize}
\end{theorem}
\noindent\emph{Sketch of proof:}\quad
We can go in the reverse direction of Okubo's construction. Given a symmetric composition algebra $(\cS,*,\norm)$ over a field containing $\omega$, define the algebra $\cA=\FF 1\oplus \cS$ with multiplication determined by the formula \eqref{eq:sym_xy}.
 Then $\cA$ turns out to be a separable alternative algebra of degree $3$.

In case $\omega\not\in\FF$, then we must consider $\cA=\FF[\omega]1\oplus\bigl(\FF[\omega]\otimes \cS\bigr)$, with the same formula for the product. In $\FF[\omega]$ we have the Galois automorphism $\omega^\tau=\omega^2$. Then the conditions $J(1)=1$ and $J(s)=-s$ for any $s\in \cS$ induce a $\FF[\omega]/\FF$-involution of the second kind in $\cA$. \qed

\begin{corollary}\label{co:symmetric_not3}
The algebras in Examples \ref{ex:Okubo} essentially exhaust, up to isomorphism, the symmetric composition algebras over a field $\FF$ of characteristic not $3$.
\end{corollary}
\noindent\emph{Sketch of proof:}\quad
Let $\omega$ be a primitive cubic root of $1$ in an algebraic closure of $\FF$, and let $\KK=\FF[\omega]$, so that $\KK=\FF$ if $\omega\in\FF$.
A separable alternative algebra over $\KK$ is, up to isomorphism, one of the following:
\begin{itemize}
\item a central simple associative algebra, and hence we obtain the Okubo algebras in Examples \ref{ex:Okubo}, 

\item $\cA=\KK\times\cC$ for a Hurwitz algebra $\cC$, in which case $(\cA_0,*,\norm)$ is shown to be isomorphic to the 
para-Hurwitz algebra attached
    to $\cC$ if $\KK=\FF$, and $\bigl(K(\cA,J)_0,*,\norm\bigr)$ to the para-Hurwitz algebra attached to $\widehat{\cC}=\{x\in \cC: J(x)=\overline{x}\}$ if $\KK\ne\FF$, 
    
\item
 $\cA=\KK\otimes_{\FF}\LL$, for a cubic field extension $\LL$ of $\FF$ (if $\omega\not\in\FF$, $\LL=\{x\in \cA: J(x)=x\}$), in which case the symmetric composition algebra is shown to be a twisted form of a two-dimensional para-Hurwitz algebra. \qedhere
\end{itemize}

One of the clues to understand symmetric composition algebras over fields of characteristic $3$ is the following result of Peterson \cite{Pet69} (dealing with $\chr\FF\neq 2,3$\,!).

\begin{theorem}\label{th:Pet}
Let $\FF$ be an algebraically closed field of characteristic $\ne 2,3$. Then any simple finite-dimensional algebra satisfying
\begin{equation}\label{eq:Pet}
(xy)x=x(yx),\quad ((xz)y)(xz)=\bigl(x((zy)z)\bigr)x
\end{equation}
for any $x,y,z$ is, up to isomorphism, one of the following:
\begin{itemize}
\item The algebra $(\cB,\bullet)$, where $(\cB,\cdot,\norm)$ is a Hurwitz algebra and $x\bullet y=\overline{x}\cdot\overline{y}$ (that is, a para-Hurwitz algebra).

\item The algebra $(\cC_s(\FF),*)$, where $\cC_s(\FF)$ is the split Cayley algebra, and $x*y=\varphi(\overline{x})\varphi^2(\overline{y})$, where $\varphi$ is a precise order $3$ automorphism of $\cC_s(\FF)$ given, in the basis in Figure \ref{fig:splitCayley} by 
\[
e_i\mapsto e_i,\ i=1,2,\qquad u_j\mapsto \omega^{j-1}u_j,\ v_j\mapsto \omega^{1-j}v_j,\ j=1,2,3,
\] 
where $\omega$ is a primitive cubic root of $1$.
\end{itemize}
\end{theorem}

Note that any symmetric composition algebra $(\cS,*,\norm)$ satisfies \eqref{eq:Pet} so the unique, up to isomorphism, Okubo algebra over an algebraically closed field of characteristic $\neq 2,3$ must be isomorphic to the last algebra in the Theorem above, and this seems to be the first appearance of these algebras in the literature! 

This motivates the next definition:

\begin{definition}[{\cite[\S 34.b]{KMRT}}]
Let $(\cC,\cdot,\norm)$ be a Hurwitz algebra, and let $\varphi\in \Aut(\cC,\cdot,\norm)$ be an automorphism with $\varphi^3=\id$.
The composition algebra $(\cC,*,\norm)$, with
\[
x*y=\varphi(\overline{x})\cdot\varphi^2(\overline{y})
\]
is called a \emph{Petersson algebra}, and denoted by $\cC_\varphi$.
\end{definition}

In case $\varphi=\id$, the Petersson algebra is the para-Hurwitz algebra associated to $(\cC,\cdot,\norm)$.

Modifying the automorphism in Theorem \ref{th:Pet}, consider the order $3$ automorphism $\varphi$ of the split Cayley algebra given by: 
\[
\varphi(e_i)=e_i,\ i=1,2,\qquad \varphi(u_j)=u_{j+1},\ \varphi(v_j)=v_{j+1}\ \text{(indices $j$ modulo $3$).}
\]
 With this automorphism we may define Okubo algebras over arbitrary fields (see \cite{EP96}).

\begin{definition}\label{df:split_Okubo}
Let $(\cC_s(\FF),\cdot,\norm)$ be the split Cayley algebra over an arbitrary field $\FF$. The Petersson algebra $\cC_s(\FF)_\varphi$ is called the \emph{split Okubo algebra} over $\FF$.

Its twisted forms (i.e., those composition algebras $(\cS,*,\norm)$ that become isomorphic to the split Okubo algebra after extending scalars to an algebraic closure) are called \emph{Okubo algebras}.
\end{definition}

In the basis in Figure \ref{fig:splitCayley}, the multiplication table of the split Okubo algebra is given in Figure \ref{fig:splitOkubo}.
\begin{figure}[h!]
\[
\vbox{\offinterlineskip
\halign{\hfil$#$\enspace\hfil&#\vreglon
 &\hfil\enspace$#$\enspace\hfil
 &\hfil\enspace$#$\enspace\hfil&#\vregleta
 &\hfil\enspace$#$\enspace\hfil
 &\hfil\enspace$#$\enspace\hfil&#\vregleta
 &\hfil\enspace$#$\enspace\hfil
 &\hfil\enspace$#$\enspace\hfil&#\vregleta
 &\hfil\enspace$#$\enspace\hfil
 &\hfil\enspace$#$\enspace\hfil&#\vreglon\cr
 &\omit\hfil\vrule width 1pt depth 4pt height 10pt
   &e_1&e_2&\omit&u_1&v_1&\omit&u_2&v_2&\omit&u_3&v_3&\omit\cr
 \noalign{\hreglon}
 e_1&&e_2&0&&0&-v_3&&0&-v_1&&0&-v_2&\cr
 e_2&&0&e_1&&-u_3&0&&-u_1&0&&-u_2&0&\cr
 &\multispan{12}{\hregletafill}\cr
 u_1&&-u_2&0&&v_1&0&&-v_3&0&&0&-e_1&\cr
 v_1&&0&-v_2&&0&u_1&&0&-u_3&&-e_2&0&\cr
 &\multispan{12}{\hregletafill}\cr
 u_2&&-u_3&0&&0&-e_1&&v_2&0&&-v_1&0&\cr
 v_2&&0&-v_3&&-e_2&0&&0&u_2&&0&-u_1&\cr
 &\multispan{12}{\hregletafill}\cr
 u_3&&-u_1&0&&-v_2&0&&0&-e_1&&v_3&0&\cr
 v_3&&0&-v_1&&0&-u_2&&-e_2&0&&0&u_3&\cr
 &\multispan{13}{\hreglonfill}\cr}}
\]
\caption{Multiplication table of the split Okubo algebra}\label{fig:splitOkubo}
\end{figure}

Over fields of characteristic $\neq 3$, our new definition of Okubo algebras coincide with the definition in Examples \ref{ex:Okubo}, due to Corollary \ref{co:symmetric_not3}. Okubo and Osborn \cite{OO81b} had given an ad hoc definition of the Okubo algebra over an algebraically closed field of characteristic $3$.

Note that the split Okubo algebra does not contain any nonzero element that commutes with every other element, that is, its commutative center is trivial. This is not so for the para-Hurwitz algebra, where the para-unit lies in the commutative center. 

Let $\FF$ be a field of characteristic $3$ and let $0\neq \alpha,\beta\in\FF$. Consider the elements 
\[
e_1\otimes\alpha^{1/3},\quad u_1\otimes\beta^{1/3}
\]
in $\cC_s(\FF)\otimes_\FF\overline{\FF}$ ($\overline{\FF}$ being an algebraic closure of $\FF$). These elements generate, by multiplication and linear combinations over $\FF$, a twisted form of the split Okubo algebra $(\cC_s(\FF),*,\norm)$. Denote by $\cO_{\alpha,\beta}$ this twisted form. 

The classification of the symmetric composition algebras in characteristic $3$, which completes the classification of symmetric composition algebras over fields, is the following \cite{Eld97} (see also \cite{CEKT}):

\begin{theorem}\label{th:symmetric_3}
Any symmetric composition algebra $(\cS,*,\norm)$ over a field $\FF$ of characteristic $3$ is either:
\begin{itemize}
\item A para-Hurwitz algebra. Two such algebras are isomorphic if and only if so are the associated Hurwitz algebras.
\item A two-dimensional algebra with a basis $\{u,v\}$ and multiplication given by
\[
u*u=v,\quad u*v=v*u=u,\quad v*v=\lambda u-v,
\]
for a nonzero scalar $\lambda\in\FF\setminus\FF^3$. These algebras do not contain idempotents and are twisted forms of the para-Hurwitz algebras. 

Algebras corresponding to the scalars $\lambda$ and $\lambda'$ are isomorphic if and only if $\FF^3\lambda+\FF^3(\lambda^2+1)=\FF^3\lambda'+\FF^3((\lambda')^2+1)$.
\item Isomorphic to $\cO_{\alpha,\beta}$ for some $0\neq\alpha,\beta\in\FF$. Moreover, $\cO_{\alpha,\beta}$ is isomorphic or antiisomorphic to $\cO_{\gamma,\delta}$ if and only if $\textrm{span}_{\FF^3}\left\{\alpha^{\pm 1},\beta^{\pm 1},\alpha^{\pm 1}\beta^{\pm 1}\right\}=\textrm{span}_{\FF^3}\left\{\gamma^{\pm 1},\delta^{\pm 1},\gamma^{\pm 1}\delta^{\pm 1}\right\}$.
\end{itemize}
\end{theorem}

A more precise statement for the isomorphism condition in the last item is given in \cite{Eld97}. A key point in the proof of this Theorem is the study of idempotents on Okubo algebras. If there are nonzero idempotents, then these algebras are Petersson algebras. The most difficult case appears in the absence of idempotents. This is only possible if the ground field $\FF$ is not perfect.

\bigskip

\section{Triality}\label{se:triality}

The importance of symmetric composition algebras lies in their connections with the phenomenon of triality in dimension $8$, related to the fact that the Dynkin diagram $D_4$ is the most symmetric one. The details of much of what follows can be consulted in \cite[Chapter VIII]{KMRT}.

Let $(\cS,*,\norm)$ be an eight-dimensional symmetric composition algebra over a field $\FF$, that is, $\cS$ is either a para-Hurwitz algebra or an Okubo algebra. Write $L_x(y)=x*y=R_y(x)$ as usual. Then, due to Theorem \ref{th:symmetric_comp}, $L_xR_x=R_xL_x=\norm(x)\id$ for any $x\in \cS$ so that, inside $\End_\FF(\cS\oplus\cS)\simeq \Mat_2\bigl(\End_\FF(\cS)\bigr)$ we have
\[
\begin{pmatrix}0&L_x\\ R_x&0\end{pmatrix}^2=\norm(x)\id.
\]
Therefore, the map 
\[
x\mapsto \begin{pmatrix}0&L_x\\ R_x&0\end{pmatrix}
\]
extends to an isomorphism of associative algebras with involution:
\[
\Phi:\left(\Cl(\cS,\norm),\tau\right)\longrightarrow \left(\End_\FF(\cS\oplus\cS),\sigma_{n\perp n}\right)
\]
where $\Cl(\cS,\norm)$ is the Clifford algebra on the quadratic space $(\cS,\norm)$, $\tau$ is its canonical involution ($\tau(x)=x$ for any $x\in\cS$), and $\sigma_{n\perp n}$ is the orthogonal involution  on $\End_\FF(\cS\oplus\cS)$ induced by the quadratic form where the two copies of $\cS$ are orthogonal and the restriction on each copy coincides with the norm. The multiplication in the Clifford algebra will be denoted by juxtaposition.

Consider the \emph{spin group}:
\[
\Spin(\cS,n)=\{ u\in \Cl(\cS,\norm)\subo^\times: u x u^{-1}\in \cS,\ u \tau(u)=1,\ \forall x\in \cC\}.
\]
For any $u\in \Spin(\cS,n)$,
\[
\Phi(u)=\begin{pmatrix} \rho_u^-&0\\ 0&\rho_u^+\end{pmatrix}
\]
for some $\rho_u^\pm\in \Ort(\cS,n)$  such that
\[
\chi_u(x*y)=\rho_u^+(x)*\rho_u^-(y)
\]
for any $x,y\in \cS$, where $\chi_u(x)=u x u^{-1}$
 gives the natural representation of $\Spin(\cS,\norm)$, while $\rho_u^{\pm 1}$ give the two half-spin representations, and the formula above links the three of them!

The last condition is equivalent to:
\[
\langle \chi_u(x),\rho_u^+(y),\rho_u^-(z)\rangle =\langle x,y,z\rangle
\]
for any $x,y,z\in \cS$, where
\[
\langle x,y,z\rangle =\norm(x,y*z),
\]
and this has cyclic symmetry:
\[
\langle x,y,z\rangle=\langle y,z,x\rangle.
\]

\begin{theorem}
Let $(\cS,*,\norm)$ be an eight-dimensional symmetric composition algebra. Then:
\[
\Spin(\cS,\norm)\simeq \{(f_0,f_1,f_2)\in \Ort^+(\cS,n)^3:
 f_0(x*y)=f_1(x)*f_2(y)\ \forall x,y\in \cS\}.
\]
Moreover, the set of \emph{related triples} (the set on the right hand side) has cyclic symmetry.
\end{theorem}

The cyclic symmetry on the right hand side induces an outer automorphism of order $3$ (\emph{trialitarian automorphism}) of $\Spin(\cS,\norm)$. Its fixed subgroup is the group of automorphisms of the symmetric composition algebra $(\cS,*,\norm)$, which is a simple algebraic group of type $G_2$ in the para-Hurwitz case, and of type $A_2$ in the Okubo case if $\chr\FF\neq 3$. 

The group(-scheme) of automorphisms of an Okubo algebra over a field of characteristic $3$ is not smooth \cite{CEKT}.

\smallskip

At the Lie algebra level, assume $\chr\FF\ne 2$, and consider the associated orthogonal Lie algebra
\[
\frso(\cS,\norm)=\{ d\in\End_{\FF}(\cS): \norm\bigl(d(x),y\bigr)+\norm\bigl(x,d(y)\bigr)=0\ \forall x,y\in \cS\}.
\]
The \emph{triality Lie algebra} of $(\cS,*,\norm)$ is defined as the following Lie subalgebra of $\frso(\cS,\norm)^3$ (with componentwise bracket):
\[
\tri(\cS,*,\norm)=\{(d_0,d_1,d_2)\in\frso(\cS,\norm)^3: d_0(x*y)=d_1(x)*y+x*d_2(y)\ \forall x,y,z\in \cS\}.
\]

Note that the condition $d_0(x*y)=d_1(x)*y+x*d_2(y)$ for any $x,y\in \cS$ is equivalent to the condition
\[
\norm\bigl(x*y,d_0(z)\bigr)+\norm\bigl(d_1(x)*y,z\bigr)+\norm\bigl(x*d_2(y),z\bigr)=0,
\]
for any $x,y,z\in \cS$. But $\norm(x*y,z)=\norm(y*z,x)=\norm(z*x,y)$. Therefore, the linear map:
\[
\begin{split}
\theta: \tri(\cS,*,\norm)&\longrightarrow \tri(\cS,*,\norm)\\
  (d_0,d_1,d_2)&\mapsto (d_2,d_0,d_1),
\end{split}
\]
is an automorphism of the Lie algebra $\tri(\cS,*,\norm)$.

\begin{theorem}
Let $(\cS,*,\norm)$ be an eight-dimensional symmetric composition algebra over a field of characteristic $\ne 2$. Then:
\begin{itemize}
\item \textup{\textbf{Principle of Local Triality:}}\quad The projection map:
\[
\begin{split}
\pi_0: \tri(\cS,*,\norm)&\longrightarrow \frso(\cS,\norm)\\
  (d_0,d_1,d_2)&\mapsto\ d_0
\end{split}
\]
is an isomorphism of Lie algebras.

\item For any $x,y\in \cS$, the triple
\[
t_{x,y}=\Bigl(\sigma_{x,y}=\norm(x,.)y-\norm(y,.)x,
\frac{1}{2}\norm(x,y)id-R_xL_y,\frac{1}{2}\norm(x,y)id-L_xR_y\Bigr)
\]
belongs to $\tri(\cS,*,\norm)$, and $\tri(\cS,*,\norm)$ is spanned by these elements. Moreover, for any $a,b,x,y\in \cS$:
\[
[t_{a,b},t_{x,y}]=t_{\sigma_{a,b}(x),y}+t_{x,\sigma_{a,b}(y)}.
\]
\end{itemize}
\end{theorem}
\begin{proof}
Let us first check that $t_{x,y}\in\tri(\cS,*,\norm)$:
\[
\begin{split}
\sigma_{x,y}(u*v)&=\norm(x,u*v)y-\norm(y,u*v)x\\[4pt]
R_xL_y(u)*v&=\bigl((y*u)*x\bigr)*v=-(v*x)*(y*u)+\norm(y*u,v)x,\\[4pt]
u*L_xR_y(v)&=u*\bigl(x*(v*y)\bigr)=-u*\bigl(y*(v*x)\bigr)+\norm(x,y)u*v\\
  &=(v*x)*(y*u)+\norm(u,v*x)y+\norm(x,y)u*v,
\end{split}
\]
and hence
\[
\sigma_{x,y}(u*v)-\Bigl(\frac{1}{2}\norm(x,y)id-R_xL_y\Bigr)(u)*v -
u*\Bigl(\frac{1}{2}\norm(x,y)id-L_xR_y\Bigr)(v)=0.
\]
Also $\sigma_{x,y}\in\frso(\cS,\norm)$ and $
\Bigl(\frac{1}{2}\norm(x,y)id-R_xL_y\Bigr)^*=\frac{1}{2}\norm(x,y)id-R_yL_x
$ (adjoint relative to the norm $\norm$), but $R_xL_x=\norm(x)id$, so
$R_xL_y+R_yL_x=\norm(x,y)id$ and hence $
\Bigl(\frac{1}{2}\norm(x,y)id-R_xL_y\Bigr)^*=-\Bigl(\frac{1}{2}\norm(x,y)id-R_xL_y\Bigr),
$ so that $\frac{1}{2}\norm(x,y)id-R_xL_y\in\frso(\cS,\norm)$, and
$\frac{1}{2}\norm(x,y)id-L_xR_y\in\frso(\cS,\norm)$ too.  Therefore,
$t_{x,y}\in\tri(\cS,*,\norm)$.

Since the Lie algebra $\frso(\cS,\norm)$ is spanned by the
$\sigma_{x,y}$'s, it is clear that  the projection $\pi_0$ is
surjective (and hence so are $\pi_1$ and $\pi_2$). It is not difficult to check that $\ker\pi_0=0$ and, therefore, $\pi_0$
is an isomorphism.

Finally the formula
$[t_{a,b},t_{x,y}]=t_{\sigma_{a,b}(x),y}+t_{x,\sigma_{a,b}(y)}$
follows from the ``same'' formula for the $\sigma$'s and the fact
that $\pi_0$ is an isomorphism.
\end{proof}

Given two symmetric composition algebras $(\cS,*,\norm)$ and $(\cS',\star,\norm')$, consider the vector space:
\[
\frg=\frg(\cS,\cS')=\bigl(\tri(\cS)\oplus\tri(\cS')\bigr)\oplus\Bigl(\oplus_{i=0}^2\iota_i(\cS\otimes \cS')\Bigr),
\]
where $\iota_i(\cS\otimes \cS')$ is just a copy of $\cS\otimes \cS'$ ($i=0,1,2$) and we write $\tri(\cS)$, $\tri(\cS')$ instead of $\tri(\cS,*,\norm)$ and $\tri(\cS',\star,\norm')$ for short. Define now an anticommutative bracket on $\frg$ by means of:

\begin{itemize}
\item the Lie bracket in $\tri(\cS)\oplus\tri(\cS')$, which thus becomes  a Lie subalgebra of $\frg$,

\item $[(d_0,d_1,d_2),\iota_i(x\otimes
 x')]=\iota_i\bigl(d_i(x)\otimes x'\bigr)$,

\item
 $[(d_0',d_1',d_2'),\iota_i(x\otimes
 x')]=\iota_i\bigl(x\otimes d_i'(x')\bigr)$,

\item $[\iota_i(x\otimes x'),\iota_{i+1}(y\otimes y')]=
 \iota_{i+2}\bigl((x* y)\otimes (x'\star y')\bigr)$ (indices modulo
 $3$),

\item $[\iota_i(x\otimes x'),\iota_i(y\otimes y')]=
 \norm'(x',y')\theta^i(t_{x,y})+
 \norm(x,y)\theta'^i(t'_{x',y'})\in\tri(\cS)\oplus\tri(\cS')$.
\end{itemize}

\begin{theorem}[{\cite{Eld04}}]
Assume $\chr\FF\neq 2,3$. With the bracket above, $\frg(\cS,\cS')$ is a Lie algebra and, if $\cS_r$ and $\cS_s'$ denote symmetric composition algebras of dimension $r$ and $s$, then the Lie algebra $\frg(\cS_r,\cS_s')$ is a (semi)simple Lie algebra whose type is given by Freudenthal's Magic Square:
$$
\vbox{\offinterlineskip
 \halign{\hfil\ $#$\quad \hfil&%
 \vreglon #%
 &\hfil\quad $#$\quad \hfil&\hfil$#$\hfil
 &\hfil\quad $#$\quad \hfil&\hfil\quad $#$\quad \hfil\cr
 \bigstrut &width 0pt&\cS_1&\cS_2&\cS_4&\cS_8\cr
 &\multispan5{\hreglonfill}\cr
 \cS_1'&&A_1&A_2&C_3&F_4\cr
 \bigstrut \cS_2'&&A_2 &A_2\oplus A_2&A_5&E_6\cr
 \bigstrut \cS_4'&&C_3&A_5 &D_6&E_7\cr
 \bigstrut \cS_8'&&F_4&E_6& E_7&E_8\cr}}
$$
\end{theorem}

Different versions of this result using Hurwitz algebras instead of symmetric composition algebras have appeared over the years (see \cite{Eld04} and the references there in). The advantage of using symmetric composition algebras is that new constructions of the exceptional simple Lie algebras are obtained, and these constructions highlights interesting symmetries due to the different triality automorphisms.

A few changes are needed for characteristic $3$. Also, quite interestingly, over fields of characteristic $3$ there are nontrivial symmetric composition \emph{superalgebras}, and these can be plugged into the previous construction to obtain an extended Freudenthal Magic Square that includes some new simple finite dimensional Lie superalgebras (see \cite{CunhaElduque06}).

\bigskip

\section{Concluding remarks}

It is impossible to give a thorough account of composition algebras in a few pages, so necessarily many things have been left out: Pfister forms and the problem of composition of quadratic forms (see \cite{Shapiro}), composition algebras over rings (or even over schemes), where Hurwitz algebras are no longer determined by their norms (see \cite{Gille}), the closely related subject of absolute valued algebras (see \cite{Angel}), ...

The interested reader may consult some excellent texts: \cite{ConwaySmith,SV00,Numbers,KMRT,Okubo_book}.
\cite{Baez} is a beautiful introduction to octonions and some of their many applications.

\smallskip

Let us finish with the first words of Okubo in his introduction to the monograph \cite{Okubo_book}:
\begin{quotation}
The saying that God is the mathematician, so that, even with meager experimental support, a mathematically beautiful theory will ultimately have a greater chance of being correct, has been attributed to Dirac. Octonion algebra may surely be called a beautiful mathematical entity. Nevertheless, it has never been systematically utilized in physics in any fundamental fashion, although some attempts have been made toward this goal. However, it is still possible that non-associative algebras (other than Lie algebras) may play some essential future role in the ultimate theory, yet to be discovered.
\end{quotation}

%-------------------------------------------

\end{document}